\documentclass[12pt, amsfonts]{amsart}

\usepackage{cite}                       
\usepackage{amsfonts,amssymb}           
\usepackage[dvipdfm,
            pdfstartview=FitH,
            CJKbookmarks=true,
            bookmarksnumbered=true,
            bookmarksopen=false,
            colorlinks,                 
            linkcolor=black,            
            anchorcolor=black,          
            citecolor=black             
            ]{hyperref}                 

\newtheorem{thm}{Theorem~}

\newtheorem{lem}[thm]{Lemma~}

\newtheorem{prop}[thm]{Proposition~}
\theoremstyle{remark}
\newtheorem{rem}[thm]{Remark~}
\numberwithin{equation}{section}

\newtheoremstyle{specthm}{1.5ex plus 1ex minus .2ex}{1.5ex plus 1ex minus .2ex}{\it}{}{\bf}{}{1em}{\thmnote{#3}}
\theoremstyle{specthm}

\def\R{{\mathbb R}}

\def\e{\varepsilon}

\newcommand{\sobo}[2]{W^{#1,#2}}                        
\newcommand{\abs}[1]{\left\vert{#1}\right\vert}         
\newcommand{\pos}[1]{\left(#1\right)_+}                 
\newcommand{\set}[1]{\left\{{#1}\right\}}               
\newcommand{\sr}{S^{n-1}}                               
\newcommand{\normk}[2]{\left\Vert{#2}\right\Vert_{#1}}  

\begin{document}


\title[Asymmetric anisotropic fractional Sobolev norms]
{Asymmetric anisotropic fractional Sobolev norms}

\author{Dan Ma}

\address{Institut f\"{u}r Diskrete Mathematik und Geometrie\\
Technische Universit\"{a}t Wien\\
Wiedner Hauptstra{\ss}e 8--10/1046, Wien 1040, Austria}

\email{madan516@gmail.com}

\date{\today}


\begin{abstract}
    Bourgain, Brezis \& Mironescu showed that (with suitable scaling)
    the fractional Sobolev $s$-seminorm of a function\\ $f\in\sobo{1}{p}(\R^n)$
    converges to the Sobolev seminorm of $f$ as $s\rightarrow1^-$.
    Ludwig introduced the anisotropic fractional Sobolev $s$-seminorms of $f$
    defined by a norm on $\R^n$ with unit ball $K$, and showed that
    they converge to the aniso\-tropic Sobolev seminorm of $f$ defined by
    the norm whose unit ball is the polar $L_p$ moment body of $K$, as $s\rightarrow1^-$. The asymmetric anisotropic $s$-seminorms are shown to
    converge to the anisotropic Sobolev seminorm of $f$ defined by the Minkowski
    functional of the polar asymmetric $L_p$ moment body of $K$.
\end{abstract}
\maketitle


\section{Introduction}

Let $\Omega$ be an open set in $\R^n$. For $p\geq1$ and $0<s<1$, Gagliardo introduced
the fractional Sobolev spaces
\[\sobo{s}{p}(\Omega)=\set{f\in L^p(\Omega):\frac{\abs{f(x)-f(y)}}{\abs{x-y}^{\frac{n}{p}+s}}\in L^p(\Omega\times\Omega)},\]
and the fractional Sobolev $s$-seminorm of a function $f\in L^p(\Omega)$
\[\normk{\sobo{s}{p}(\Omega)}{f}^p=\int\limits_\Omega\int\limits_\Omega\frac{\abs{f(x)-f(y)}^p}{\abs{x-y}^{n+ps}}dxdy\]
(see \cite{Gag57}). They have found many applications in pure and applied mathematics (see \cite{BBM02b,DPV12,Maz11}).

Although $\normk{\sobo{s}{p}(\Omega)}{f}\rightarrow\infty$ as $s\rightarrow1^-$,
Bourgain, Brezis and Mironescu showed in \cite{BBM02} that
\begin{equation}\label{eqn:BBM}
    \lim_{s\rightarrow1^-}(1-s)\normk{\sobo{s}{p}(\Omega)}{f}^p=\frac{K_{n,p}}{p}\normk{\sobo{1}{p}(\Omega)}{f}^p,
\end{equation}
for $f\in\sobo{1}{p}(\Omega)$ and $\Omega\subset\R^n$ a smooth bounded domain,
where
\[K_{n,p}=\frac{2\Gamma((p+1)/2)\pi^{(n-1)/2}}{\Gamma((n+p)/2)}\]
is a constant depending on $n$ and $p$,
\[\normk{\sobo{1}{p}}{f}^p=\int\limits_\Omega\abs{\nabla f(x)}^pdx\]
is the Sobolev seminorm of $f$, and $\nabla f:\R^n\rightarrow\R^n$ denotes
the $L^p$ weak derivative of $f$.

If instead of the Euclidean norm $\abs{\cdot}$, we consider an arbitrary norm $\normk{K}{\cdot}$ with unit ball $K$, we obtain the anisotropic Sobolev seminorm,
\[\normk{\sobo{1}{p},K}{f}^p=\int\limits_{\R^n}\normk{K^\ast}{\nabla f(x)}^pdx,\]
where $K^\ast=\set{v\in\R^n: v\cdot x\leq1\mbox{ for all }x\in K}$
is the polar body of $K$, and $v\cdot x$ denotes the inner product between $v$ and $x$.
Anisotropic Sobolev seminorms and the corresponding
Sobolev inequalities attracted a lot of attentions in recent years (see \cite{AFTL97,CNV04,FMP13,Gro86}).

Anisotropic $s$-seminorms, introduced very recently by Ludwig \cite{Lud13b}, reflect a fine structure of the anisotropic fractional Sobolev spaces. She established that
\[\lim_{s\rightarrow1^-}(1-s)\int\limits_{\R^n}\int\limits_{\R^n}\frac{\abs{f(x)-f(y)}^p}{\normk{K}{x-y}^{n+ps}}dxdy=\frac{2}{p}\int\limits_{\R^n}\normk{Z_p^\ast K}{\nabla f(x)}^pdx,\]
for $f\in\sobo{1}{p}(\R^n)$ with compact support, where the norm associated with $Z_p^\ast K$, the polar $L_p$ moment body of $K$, is defined as
\[\normk{Z_p^\ast K}{v}^p=\frac{n+p}{2}\int\limits_K\abs{v\cdot x}^pdx,\]
for $v\in\R^n$, and a convex body $K\subset\R^n$.
Several different other cases were considered in \cite{Lud13a,Lud13b,Pon04}.

In this paper, by replacing the absolute value $\abs{\cdot}$ by the positive part $\pos{\cdot}$, for $x\in\R$, where $\pos{x}=\max\set{0,x}$,
we obtain the following generalization. Note that here it is no longer required that $K$ is origin-symmetric. As a consequence, for $K\subset\R^n$ a convex body containing
the origin in its interior and $x\in\R^n$,
\[\normk{K}{x}=\min\set{\lambda\geq0:x\in\lambda K}\]
just defines the Minkowski functional of $K$ and no longer a norm.

\begin{thm}\label{thm:main}
    If $f\in\sobo{1}{p}(\R^n)$ has compact support, then
    \[\lim_{s\rightarrow1^-}(1-s)\int\limits_{\R^n}\int\limits_{\R^n}\frac{\pos{f(x)-f(y)}^p}{\normk{K}{x-y}^{n+sp}}dxdy=\frac{1}{p}\int\limits_{\R^n}\normk{Z_p^{+,*}K}{\nabla f(x)}^pdx,\]
    where $Z_p^{+,*}K$ is the polar asymmetric $L_p$ moment body of $K$.
\end{thm}

For a convex body $K\subset\R^n$, the polar asymmetric $L_p$ moment body is the unit ball of the Minkowski functional defined by
\[\normk{Z_p^{+,\ast}K}{v}^p=(n+p)\int\limits_K\pos{v\cdot x}^pdx,\]
for $v\in\R^n$, $Z_p^-K=Z_p^+(-K)$.
For $p>1$, in \cite{Lud05}, Ludwig introduced and characterized the two-parameter family
\[c_1\cdot Z_p^+K+_pc_2\cdot Z_p^-K\]
as all possible $L_p$ analogs of moment bodies,
including the symmetric case
\[Z_pK=\frac{1}{2}\cdot Z_p^+K+_p\frac{1}{2}\cdot Z_p^-K,\]
where $\normk{(\alpha\cdot K+_p\beta\cdot L)^\ast}{\cdot}^p=\alpha\normk{K^\ast}{\cdot}^p+\beta\normk{L^\ast}{\cdot}^p$,
for $\alpha,\beta\geq0$, defines the $L_p$ Minkowski combination.
In recent years, this family of convex bodies have found important applications
within convex geometry, probability theory, and the local theory of Banach spaces
(see \cite{Gar06,Hab12b,HS09a,Lud03,Lud05,Lud10b,Lut90,LYZ00a,LYZ00b,LYZ02b,LYZ04c,LYZ10b,Pao06a,PW12,Par13a,Par13b,Wan11}).

The proof given in this paper makes use of an asymmetric version of the one-dimensional case of result (\ref{eqn:BBM}) by Bourgain, Brezis and Mironescu and an asymmetric decomposition of Blaschke-Petkantschin type.


\section{Proof of the main result}

First, we need the asymmetric one-dimensional analogue of (\ref{eqn:BBM}).
For its proof we require the following result from \cite{BBM02}.

\begin{lem}\label{thm:est}
    Let $\rho\in L^1(\R^n)$ and $\rho\geq0$.
    If $f\in\sobo{1}{p}(\R^n)$ is compactly supported
    and $1\leq p<\infty$, then
    \[\int\limits_{\R^n}\int\limits_{\R^n}\frac{\abs{f(x)-f(y)}^p}{\abs{x-y}^p}\rho(x-y)dxdy\leq C\normk{\sobo{1}{p}}{f}^p\normk{L^1}{\rho}\]
    where $C$ depends only on $p$ and the support of $f$.
\end{lem}

Let $\Omega\subset\R$ be a bounded domain.

\begin{prop}\label{thm:1dim}
    If $f\in\sobo{1}{p}(\Omega)$, then
    \begin{equation}\label{eqn:1dimo}
        \lim_{s\rightarrow1^-}(1-s)\int\limits_\Omega\int\limits_{\Omega\cap\set{x>y}}\frac{\pos{f(x)-f(y)}^p}{\abs{x-y}^{1+ps}}dxdy=\frac{1}{p}\int\limits_\Omega\pos{f'(x)}^pdx.
    \end{equation}
\end{prop}

\begin{proof}
    Take a sequence $(\rho_\e)$ of radial mollifiers, i.e. $\rho_\e(x)=\rho_\e(\abs{x})$;
    $\rho_\e\geq0$; $\int_0^\infty\rho_\e(x)dx=1$; $\lim\limits_{\e\rightarrow0}\int_\delta^\infty\rho_\e(r)dr=0$ for every $\delta>0$.
    Let $F_\e(x,y)=\frac{\pos{f(x)-f(y)}}{\abs{x-y}}\rho_\e^{1/p}(x-y)$, for $x>y$. It suffices to prove that
    \begin{equation}\label{eqn:res}
        \lim_{\e\rightarrow0}\int\limits_\Omega\int\limits_{\Omega\cap\set{x>y}}{F_\e}^p(x,y)dxdy=\int\limits_\Omega\pos{f'(x)}^pdx.
    \end{equation}
    Indeed, as in \cite{Spe11}, let $R>\max\set{\abs{x-y}:x,y\in\Omega}$, $\e=1-s$ and
    \[\rho_\e(x)=\frac{\chi_{[0,R]}(\abs{x})}{R^{\e p}}\frac{p\e}{\abs{x}^{1-p\e}},\]
    where $\chi_A$ is the indicator function of $A$.
    Then one obtains (\ref{eqn:1dimo}) from \eqref{eqn:res} as desired.

    By Lemma \ref{thm:est}, we have, for any $\e>0$ and $f,g\in\sobo{1}{p}(\Omega)$
    \[\abs{\normk{L^p(\Omega\times\Omega)}{F_\e}-\normk{L^p(\Omega\times\Omega)}{G_\e}}
    \leq\normk{L^p(\Omega\times\Omega)}{F_\e-G_\e}\leq C\normk{\sobo{1}{p}}{f-g},\]
    for some constant $C$ dependent on $\e, f$ and $g$. Therefore,
    it suffices to establish (\ref{eqn:res}) for $f$ in some dense subset of $\sobo{1}{p}(\Omega)$, e.g., for $f\in C^2(\bar{\Omega})$, where $\bar{\Omega}$ is the closure of $\Omega$.

    Fix $f\in C^2(\bar{\Omega})$. Since for $t\in\R$ and $\lambda>0$, $\pos{\lambda t}=\lambda\pos{t}$, there exists $\delta>0$, such that for $y<x<y+\delta$ and a constant c,
    \[\abs{\frac{\pos{f(x)-f(y)}^p}{\abs{x-y}^p}-\pos{f'(y)}^p}\leq c(x-y).\]
    We have
    \begin{align*}
        &\int\limits_{\Omega\cap\set{x>y}}\frac{\pos{f(x)-f(y)}^p}{\abs{x-y}^p}\rho_\e(x-y)dx\\
        =& \int\limits_{\Omega\cap\set{y<x<y+\delta}}\frac{\pos{f(x)-f(y)}^p}{\abs{x-y}^p}\rho_\e(x-y)dx
        +\int\limits_{\Omega\cap\set{x\geq y+\delta}}\frac{\pos{f(x)-f(y)}^p}{\abs{x-y}^p}\rho_\e(x-y)dx,
    \end{align*}
    yet, only the former integral on the right hand side need be considered, as the latter vanishes.
    In fact, for each fixed $y\in\Omega$, since
    \begin{align*}
        & \abs{\int\limits_y^{y+\delta}\left(\frac{\pos{f(x)-f(y)}^p}{\abs{x-y}^p}-\pos{f'(y)}^p\right)\rho_\e(x-y)dx}\\
        \leq& \int\limits_y^{y+\delta}\abs{\frac{\pos{f(x)-f(y)}^p}{\abs{x-y}^p}-\pos{f'(y)}^p}\rho_\e(x-y)dx\\
        \leq& c\int\limits_y^{y+\delta}(x-y)\rho_\e(x-y)dx\\
        =& c\int\limits_0^\delta r\rho_\e(r)dr\rightarrow0,\mbox{ as }\e\rightarrow0,
    \end{align*}
    we have
    \begin{align*}
        & \lim_{\e\rightarrow0}\int\limits_y^{y+\delta}\frac{\pos{f(x)-f(y)}^p}{\abs{x-y}^p}\rho_\e(x-y)dx\\
        =& \pos{f'(y)}^p\lim_{\e\rightarrow0}\int\limits_y^{y+\delta}\rho_\e(x-y)dx\\
        =& \pos{f'(y)}^p\lim_{\e\rightarrow0}\int\limits_0^\delta\rho_\e(r)dr\\
        =& \pos{f'(y)}^p.
    \end{align*}
    Therefore,
    \begin{equation}\label{eqn:lim}
        \lim_{\e\rightarrow0}\int\limits_{\Omega\cap\set{x>y}}\frac{\pos{f(x)-f(y)}^p}{\abs{x-y}^p}\rho_\e(x-y)dx=\pos{f'(y)}^p.
    \end{equation}
    Since $f\in C^2(\bar{\Omega})$, there exists $L>0$ is such that $\abs{f(x)-f(y)}<L\abs{x-y}$, for every $x,y\in\Omega$, then
    \begin{equation}\label{eqn:bnd}
        \int\limits_{\Omega}\frac{\abs{f(x)-f(y)}^p}{\abs{x-y}^p}\rho_\e(x-y)dx\leq L^p,\quad\mbox{for each }y\in\Omega.
    \end{equation}
    Hence, for $f\in C^2(\Omega)$, \eqref{eqn:res} follows by dominated convergence theorem from \eqref{eqn:lim} and \eqref{eqn:bnd}.
\end{proof}

Now, for $u\in S^{n-1}$, the Euclidean unit sphere, let $[u]=\set{\lambda u:\lambda\in\R}$ and $[u]^+=\set{\lambda u:\lambda>0}$. Denote the $k-$dimensional Hausdorff measure on $\R^n$ by $H^k$.
For $f\in\sobo{1}{p}(\R^n)$, we denote by $\bar{f}$ its precise representative (see \cite[Section 1.7.1]{EG92}). We require the following result. For every $u\in\sr$,
the precise representative $\bar{f}$ is absolutely continuous on the lines $L=\set{x+\lambda u:\lambda\in\R}$ for $H^{n-1}-$a.e. $x\in u^\bot$ and its first-order
(classical) partial derivatives belong to $L^p(\R^n)$ (see \cite[Section 4.9.2]{EG92}).
Hence, we have for the restriction of $\bar{f}$ to $L$,
\begin{equation}\label{eqn:res1dim}
    \bar{f}\big|_L\in\sobo{1}{p}(L)
\end{equation}
for a.e. line $L$ parallel to $u$.

\begin{proof}[Proof of Theorem \ref{thm:main}]
    By the polar coordinate formula and Fubini's theorem, we have
    \begin{equation}\label{eqn:BP}
    \begin{split}
        &\int\limits_{\R^n}\int\limits_{\R^n}\frac{\pos{f(x)-f(y)}^p}{\normk{K}{x-y}^{n+sp}}dH^n(x)dH^n(y)\\
        =& \int\limits_{\R^n}\int\limits_{\sr}\normk{K}{u}^{-(n+ps)}\int\limits_0^\infty\frac{\pos{f(y+ru)-f(y)}^p}{r^{1+sp}}dH^1(r)d\sigma(u)dH^n(y)\\
        =& \int\limits_{\sr}\normk{K}{u}^{-(n+ps)}\int\limits_0^\infty\int\limits_{u^\bot}\int\limits_{[u]+z}\frac{\pos{f(w+ru)-f(w)}^p}{r^{1+sp}}dH^1(w)dH^{n-1}(z)dH^1(r)d\sigma(u)\\
        =& \int\limits_{\sr}\normk{K}{u}^{-(n+ps)}\int\limits_{u^\bot}\int\limits_{[u]+z}\int\limits_0^\infty\frac{\pos{f(w+ru)-f(w)}^p}{r^{1+sp}}dH^1(r)dH^1(w)dH^{n-1}(z)d\sigma(u)\\
        =& \int\limits_{\sr}\normk{K}{u}^{-(n+ps)}\int\limits_{u^\bot}\int\limits_{[u]+z}\int\limits_{[u]^++w}\frac{\pos{f(t)-f(w)}^p}{\abs{t-w}^{1+sp}}dH^1(t)dH^1(w)dH^{n-1}(z)d\sigma(u),
    \end{split}
    \end{equation}
    where $\sigma$ denotes the standard surface area measure on $\sr$.
    By Proposition \ref{thm:1dim} and (\ref{eqn:res1dim}), we obtain
    \begin{equation}\label{eqn:1dim}
    \begin{split}
        &\lim_{s\rightarrow1^-}(1-s)\int\limits_{[u]+z}\int\limits_{[u]^++w}\frac{\pos{f(t)-f(w)}^p}{\abs{t-w}^{1+sp}}dH^1(t)dH^1(w)\\
        =&\frac{1}{p}\int\limits_{[u]+z}\pos{\nabla f(t)\cdot u}^pdH^1(t).
    \end{split}
    \end{equation}

    By Fubini's theorem and the polar coordinate formula,
    we get
    \begin{align*}
        & \frac{1}{p}\int\limits_{\sr}\normk{K}{u}^{-(n+p)}\int\limits_{u^\bot}\int\limits_{[u]+z}\pos{\nabla f(t)\cdot u}^pdH^1(t)dH^{n-1}(z)d\sigma(u)\\
        =& \frac{1}{p}\int\limits_{\sr}\int\limits_{\R^n}\normk{K}{u}^{-(n+p)}\pos{\nabla f(x)\cdot u}^pdH^n(x)d\sigma(u)\\
        =& \frac{n+p}{p}\int\limits_K\int\limits_{\R^n}\pos{\nabla f(x)\cdot y}^pdH^n(x)dH^n(y)
    \end{align*}
    Using Fubini's theorem and the definition of the asymmetric $L_p$ moment body of $K$, we obtain
    \begin{equation}\label{eqn:inner}
    \begin{split}
        &\int\limits_{\sr}\normk{K}{u}^{-(n+p)}\int\limits_{u^\bot}\int\limits_{[u]+z}\pos{\nabla f(t)\cdot u}^pdH^1(t)dH^{n-1}(z)d\sigma(u)\\
        =&\int\limits_{\R^n}\normk{Z_p^{+,*}K}{\nabla f(x)}^pdH^n(x).
    \end{split}
    \end{equation}
    So, in particular, we have
    \begin{equation}\label{eqn:innerest}
    \begin{split}
        &\int\limits_{\sr}\int\limits_{u^\bot}\int\limits_{[u]+z}\pos{\nabla f(t)\cdot u}^pdH^1(t)dH^{n-1}(z)d\sigma(u)\\
        =&\frac{n+p}{4}K_{n,p}\int\limits_{\R^n}\abs{\nabla f(x)}^pdH^n(x)<+\infty.
    \end{split}
    \end{equation}

    Using the dominated convergence theorem with Lemma \ref{thm:est} and (\ref{eqn:innerest}), we obtain from (\ref{eqn:BP}), (\ref{eqn:1dim}) and (\ref{eqn:inner}) that
    \[\lim_{s\rightarrow1^-}(1-s)\int\limits_{\R^n}\int\limits_{\R^n}\frac{\pos{f(x)-f(y)}^p}{\normk{K}{x-y}^{n+sp}}dxdy=\frac{1}{p}\int\limits_{\R^n}\normk{Z_p^{+,*}K}{\nabla f(x)}^pdx.\]
\end{proof}

\begin{rem}
    In Theorem \ref{thm:main}, let $g=-f$ and $(x)_-=-\min\set{0,x}=\pos{-x}$, for $x\in\R$.
    Then, we get
    \[\lim_{s\rightarrow1^-}(1-s)\int\limits_{\R^n}\int\limits_{\R^n}\frac{(f(x)-f(y))_-^p}{\normk{K}{x-y}^{n+sp}}dxdy=\frac{1}{p}\int\limits_{\R^n}\normk{Z_p^{-,*}K}{\nabla f(x)}^pdx.\]
\end{rem}

\bigskip

\textbf{Acknowledgement.} The author wishes to thank Professor Monika\\ Ludwig
for her supervision and valuable suggestion, and the referee for\\ valuable suggestions
and careful reading on the original manuscript.
The work of the author was supported in part by Austrian Science Fund (FWF) Project P23639-N18
and National Natural Science Foundation of China Grant No. 11371239.



\begin{thebibliography}{10}

\bibitem{AFTL97}
A.~Alvino, V.~Ferone, G.~Trombetti, and P.-L. Lions, \emph{Convex
  symmetrization and applications}, Ann. Inst. H. Poincar\'{e} Anal. Non
  Lin\'{e}aire \textbf{14} (1997), 275--293.

\bibitem{BBM02}
J.~Bourgain, H.~Brezis, and P.~Mironescu, \emph{Another look at Sobolev
  spaces}, Optimal Control and Partial Differential Equations, A volume in
  honor of A. Bensoussans's 60th birthday (Amsterdam) (J.~L. Menaldi,
  E.~Rofman, and A.~Sulemn, eds.), IOS Press, 2001, pp.~439--455.

\bibitem{BBM02b}
\bysame, \emph{Limiting embedding theorems for $W^{s.p}$ when $s\uparrow1$ and
  applications}, J. Anal. Math. \textbf{87} (2002), 77--101, Dedicated to the
  memory of Thomas H. Wolff.

\bibitem{CNV04}
D.~Cordero-Erausquin, B.~Nazaret, and C.~Villani, \emph{A mass-transportation
  approach to sharp Sobolev and Gagliardo-Nirenberg inequalities}, Adv. Math.
  \textbf{182} (2004), 307--332.

\bibitem{DPV12}
E.~Di Nezza, G.~Palatucci, and E.~Valdinoci, \emph{Hitchhiker's guide to the
  fractional Sobolev spaces}, Bull. Sci. Math. \textbf{136} (2012), 521--573.

\bibitem{EG92}
L.~Evans and R.~Gariepy, \emph{Measure theory and fine properties of
  functions}, Studies in Advanced Mathematics, CRC Press, Boca Raton, FL, 1992.

\bibitem{FMP13}
A.~Figalli, F.~Maggi, and A.~Pratelli, \emph{Sharp stability theorems for the
  anisotropic Sobolev and log-Sobolev inequalities on functions of bounded
  variation}, Adv. Math. \textbf{242} (2013), 80--101.

\bibitem{Gag57}
E.~Gagliardo, \emph{Caratterizzazioni delle tracce sulla frontiera relative ad
  alcune classi di funzioni in $n$ variabili}, Rend. Sem. Mat. Univ. Padova
  \textbf{27} (1957), 284--305.

\bibitem{Gar06}
R.~J. Gardner, \emph{Geometric tomography}, 2nd ed., Cambridge Univ. Press, New
  York, 2006.

\bibitem{Gro86}
M.~Gromov, \emph{Isoperimetric inequalities in Riemannian manifolds},
  Asymptotic Theory of Finite-dimensional Normed Spaces (V.~D. Milman and
  G.~Schechtman, eds.), Springer-Verlag, Berlin Heidelberg, 1986, pp.~114--129.

\bibitem{Hab12b}
C.~Haberl, \emph{Minkowski valuations intertwining the special linear group},
  J. Eur. Math. Soc. \textbf{14} (2012), 1565--1597.

\bibitem{HS09a}
C.~Haberl and F.~Schuster, \emph{General~$L_p$ affine isoperimetric
  inequalities}, J. Differential Geom. \textbf{83} (2009), 1--26.

\bibitem{Lud03}
M.~Ludwig, \emph{Ellipsoids and matrix valued valuations}, Duke Math. J.
  \textbf{119} (2003), 159--188.

\bibitem{Lud05}
\bysame, \emph{Minkowski valuations}, Trans. Amer. Math. Soc. \textbf{357}
  (2005), 4191--4213.

\bibitem{Lud10b}
\bysame, \emph{Minkowski areas and valuations}, J. Differential Geom.
  \textbf{86} (2010), 133--161.

\bibitem{Lud13a}
\bysame, \emph{Anisotropic fractional perimeters}, J. Differential Geom.
  \textbf{96} (2014), 77--93.

\bibitem{Lud13b}
\bysame, \emph{Anisotropic fractional Sobolev norms}, Adv. Math.
  \textbf{252} (2014), 150--157.

\bibitem{Lut90}
E.~Lutwak, \emph{Centroid bodies and dual mixed volumes}, Proc. London Math.
  Soc. \textbf{60} (1990), 365--391.

\bibitem{LYZ00a}
E.~Lutwak, D.~Yang, and G.~Zhang, \emph{$L_p$ affine isoperimetric
  inequalities}, J. Differential Geom. \textbf{56} (2000), 111--132.

\bibitem{LYZ00b}
\bysame, \emph{A new ellipsoid associated with convex bodes}, Duke. Math. J.
  \textbf{104} (2000), 375--390.

\bibitem{LYZ02b}
\bysame, \emph{The Cramer-Rao inequality for star bodies}, Duke Math. J.
  \textbf{112} (2002), 59--81.

\bibitem{LYZ04c}
\bysame, \emph{Moment-entropy inequalities}, Ann. Probab. \textbf{32} (2004),
  757--774.

\bibitem{LYZ10b}
\bysame, \emph{Orlicz centroid bodies}, J. Differential Geom. \textbf{84}
  (2010), 365--387.

\bibitem{Maz11}
V.~G. Maz'ya, \emph{Sobolev spaces with applications to elliptic partial
  differential equations}, augmented ed., Grundlehren der Mathematischen
  Wissenschaften, vol. 342, Springer-Verlag, Berlin Heidelberg, 2011.

\bibitem{Pao06a}
G.~A. Paouris, \emph{Concentration of mass on convex bodies}, Geom. Funct.
  Anal. \textbf{16} (2006), 1021--1049.

\bibitem{PW12}
G.~A. Paouris and E.~Werner, \emph{Relative entropy of cone measures and $L_p$
  centroid bodies}, Proc. Lond. Math. Soc. \textbf{104} (2012), no.~2,
  253--286.

\bibitem{Par13a}
L.~Parapatits, \emph{SL(n)-contravariant $L_p$-Minkowski valuations}, Trans. Amer.
  Math. Soc. \textbf{366} (2014), 1195--1211.

\bibitem{Par13b}
\bysame, \emph{SL(n)-covariant $L_p$-Minkowski valuations}, J. London
  Math. Soc. \textbf{89} (2014), 397--414.


\bibitem{Pon04}
A.~Ponce, \emph{A new approach to Sobolev spaces and connections to
  $\Gamma$-convergence}, Calc. Var. Partial Differential Equations \textbf{19}
  (2004), 229--255.

\bibitem{Spe11}
D.~Spector, \emph{Characterization of Sobolev and BV spaces}, Ph.D. thesis,
  Carnegie Mellon University, 2011.

\bibitem{Wan11}
T.~Wannerer, \emph{GL(n) equivariant Minkowski valuations}, Indiana Univ. Math.
  J. \textbf{60} (2011), 1655--1672.

\end{thebibliography}

\end{document}